\documentclass[a4paper,11pt]{article}
\usepackage{amsmath}
\usepackage{amsthm}
\usepackage{amsfonts}
\usepackage{amssymb}
\usepackage{amsopn}
\usepackage{dsfont}
\usepackage{mathrsfs}
\usepackage{mathtools}
\usepackage[T1]{fontenc} % codifica dei font in uscita
\usepackage[utf8]{inputenc} % lettere accentate da tastiera
\usepackage[english]{babel} % lingua principale del documento
\usepackage{lipsum} % genera testo fittizio
\usepackage{url} % per scrivere gli indirizzi Internet

\input{epsf.sty}

\theoremstyle{plain}
\newtheorem{theorem}{Theorem}[section]                                          
\newtheorem{proposition}[theorem]{Proposition}                          

\newtheorem{corollary}[theorem]{Corollary}

\theoremstyle{definition}

\theoremstyle{remark}
\newtheorem{remark}[theorem]{Remark}

\DeclareMathAlphabet{\mathpzc}{OT1}{pzc}{m}{it}

\newfont{\rsfsten}{rsfs10 scaled 1200}

\begin{document}
\title{\huge \bf A fractional counting process and its connection with the Poisson process}

\author{
{\bf Antonio Di Crescenzo, \ Barbara Martinucci, \ Alessandra Meoli }\\
Dipartimento di Matematica, Universit\`a degli Studi di Salerno\\
Via Giovanni Paolo II, n.\ 132 -  I-84084 Fisciano (SA), Italy\\
Email: \{adicrescenzo, bmartinucci, ameoli\}@unisa.it  
}

\date{\bf To appear in \ {\em ALEA Latin American Journal of Probability and Mathematical Statistics}}

\maketitle

\begin{abstract}
We consider a fractional counting process with jumps of amplitude 
$1,2,\ldots,k$, with $k\in \mathbb{N}$, whose probabilities 
satisfy a suitable system of fractional difference-differential equations.
We obtain the moment generating function and the probability law of the resulting process 
in terms of generalized Mittag-Leffler functions. We also discuss two equivalent representations 
both in terms of a compound fractional Poisson process and of a subordinator governed by 
a suitable fractional Cauchy problem. The first occurrence time of a jump of fixed amplitude  
is proved to have the same distribution as the waiting time of the first event of a classical 
fractional Poisson process, this extending a well-known property of the Poisson process. 
When $k=2$ we also express the distribution of the first passage time of the fractional 
counting process in an integral form. Finally, we show that the ratios given by the powers of 
the fractional Poisson process and of the counting process over their means tend to 1 in probability.

 
\end{abstract}

%%%%%%%%%%%%%%%%
\section{Introduction and background}\label{Section1}
%%%%%%%%%%%%%%%%%
Fractional Poisson processes and related counting processes are attracting the attention of several 
authors. Most of the recent papers on this topic are centered on certain fractional versions 
(time-fractional, space-fractional, space-time fractional) of the Poisson process, as well as some 
fractional birth processes (see, for instance, the review in \cite{Orsingher2013} and \cite{Alipour2015}). 
Moreover, \cite{BeghinOrsingher2010} study the properties of 
Poisson-type fractional processes, governed by fractional recursive 
differential equations, obtained substituting regular derivatives with fractional derivatives. 
\cite{Mainardi2004} provide a generalization of the pure and compound Poisson 
processes via fractional calculus, by resorting to a renewal process-based approach involving 
waiting time distributions expressed in term of the Mittag-Leffler function. 
A different approach has been developed by \cite{Laskin2003} and \cite{Laskin2009}, 
where a fractional non-Markov Poisson stochastic process based on a fractional generalization 
of the Kolmogorov-Feller equations, and some interesting applications including a fractional compound 
Poisson process have been considered. 
More recently, \cite{Meerschaert2011} show that a Poisson process, 
with the time variable replaced by an independent inverse stable subordinator, is also a fractional 
Poisson process.  Other recent results on fractional Poisson process can be found in 
\cite{GoMa2012} and \cite{GoMa2013}. 
\\
\par
Counting processes with jumps of amplitude larger than 1 are employed in various applications, since 
they are useful to describe simultaneous but independent Poisson streams (see \cite{Adelson1966}, 
for instance). The case of fractional compound Poisson processes has been 
investigated by \cite{Scalas2011}, \cite{BeghinMacci2012} and \cite{BeghinMacci2016}, 
for instance. Moreover, \cite{BeghinMacci2014}    
consider two fractional versions of nonnegative, integer-valued compound Poisson processes, 
and prove that their probability mass function solve certain fractional Kolmogorov forward equations. 
Certain fractional growth processes including the possibility of jumps of amplitude 
larger than 1 have been obtained recently through the interesting space-fractional Poisson process 
(cf.\ \cite{OrsingherPolito2012}) and, more generally, through the class of point processes 
studied in \cite{OrsingherToaldo2015} and \cite{Polito2016}. 
The relevance of fractional compound Poisson processes in applications in ruin theory and their 
long-range dependence are investigated in \cite{BiardSaussereau2014} and \cite{Maheshwari2016}. 
\\
\par
Following the lines of the papers above, here we analyse a suitable extension of the 
fractional Poisson process, say $ M^{\nu}(t) $, which performs $k$ kinds of jumps of amplitude $1, 2, \ldots,k$
with rates $\lambda_1,\lambda_2, \ldots,\lambda_k$ respectively. 
(Throughout the paper we refer to the fractional derivative in the Caputo sense, also known as 
Dzherbashyan-Caputo fractional derivative). 
We first obtain the moment generating function and  the 
probability law of the process, and discuss its equivalent representation 
in terms of a subordinator governed by a suitable fractional Cauchy problem. 
\\
\par
Along the same lines as \cite{BeghinOrsingher2010}, in Section \ref{Section3} we 
consider the difference-differen\-tial equations governing the probability mass function of 
$ M^{\nu}(t)$ and involving the time-fractional derivative of order $ \nu\in(0,1] $. 
The solution of the resulting Cauchy problem represents the 
probability distribution of the fractional counting process  $ M^{\nu}(t) $. Hence, we obtain 
$\mathbb{E}\left[e^{s M^{\nu}(t)} \right]$ and 
$ p_{\, k}^{\, \nu }(t)=\mathbb{P}\left \{ M^{\nu}\left ( t \right )= k \right \} $ in terms of a generalized Mittag-Leffler function.
We also show two useful representations for $ M^{\nu}(t) $: \\
(i)  We prove that $ M^{\nu}(t) $ can be expressed as a compound fractional Poisson process. This representation 
is essential to obtain a waiting time distribution. \\
(ii)  We show that $ M^{\nu}(t) $ can be regarded as a homogeneous Poisson process with $k$ kinds of jumps stopped at a random time. Such random time is the sole component of this subordinating relationship affected by the fractional derivative, since its distribution is obtained from the fundamental solution of a fractional diffusion equation. \\
\par
In Section \ref{Section4} we  face the problem of determining certain waiting time and 
first-passage-time distributions. Specifically, we evaluate the probability that the first jump 
of size $ j$, $j=1,2, \ldots, k$, for the process $ M^{\nu}(t) $ occurs before time $ t>0$. 
Interestingly, we prove that the first occurrence time of a jump of amplitude $j$ has the same 
distribution as the waiting time of the first event of the classical fractional Poisson 
process defined with parameter $\lambda_j$, for $j\in\{1,2,\ldots,k\}$. This  is an immediate extension of a 
well-known result. Indeed, for a Poisson process with intensity $\lambda_1+\lambda_2$ and such that its events 
are classified as type $j$ via independent Bernoulli trials with probability $\frac{\lambda_j}{\lambda_1+\lambda_2}$, 
the first occurrence time of an event of type $j$ is distributed as the interarrival time of a Poisson 
process  with intensity $\lambda_j$, $j=1,2$. In Theorem \ref{th:WaitingTime} we extend this result to the 
fractional setting. The remarkable difference is that the exponential density of the interarrival times of the Poisson 
process is replaced by the corresponding density of the fractional Poisson process, which depends 
on the (two-parameter) Mittal-Leffler function. In Section \ref{Section4} we also study, when $k=2$, the distribution 
of the first passage time of $ M^{\nu}(t) $ to a fixed level. We express it in an integral 
form which involves the joint distribution of the fractional Poisson process.
\\
\par
Finally, in Section \ref{Section5} we  obtain a formal expression of the moments of $ M^{\nu}(t) $, and 
show that both the ratios given by the powers of the fractional Poisson process and of the 
process $ M^{\nu}(t) $ over their means tend to 1 in probability. This result 
is useful in some applications. In fact, from a physical point of view, it means that the distance between 
the distributions of such processes at time $ t $ and their equilibrium measures is close to 1 until some 
deterministic `{cutoff time}' and is close to 0 shortly after. 
%
%%%%%%%%%%%%%%%%%%%
% \section{Background on the fractional Poisson process}\label{Section2}
%%%%%%%%%%%%%%%%%%%%
\\
\par
In the remaining part of this section we briefly recall some well-known results on the fractional Poisson 
process which will be used throughout the paper. 
Consider the fractional Poisson process
\begin{equation}\label{PoissonFractional}
\left \{ N_{\,\lambda}^{\,\nu }(t) ; t\geq 0\right \},\qquad\nu\in(0,1],\,\lambda\in(0,\infty),
\end{equation}
namely the renewal process with i.i.d.\ interarrival times $ \text{\rsfsten U}_{j} $ distributed according to the 
following density, for $ j=1,2,\ldots $ and $ t\in(0,\infty)$ (see \cite{BeghinOrsingher2010}): 
\begin{equation}\label{InterArrivalDensity}
f_{\,1}^{\,\nu }\left ( t \right )= \mathbb{P}\left \{ \text{\rsfsten U}_{j}\in d\,t \right \}/d\,t=\lambda t^{\nu -1}E_{\nu ,\nu }(-\lambda t^{\nu }),
\end{equation}
where 
\begin{equation*} %MittagLeffler
E_{\alpha,\beta}(x)=\sum_{r=0}^{\infty}\frac{x^{r}}{\Gamma(\alpha r+\beta )}, \qquad \alpha,\beta \in \mathbb{C},\,Re(\alpha),Re(\beta)>0,\,x\in\mathbb{R}
\end{equation*}
is the (two-parameter) Mittag-Leffler function. From the Laplace transform
\begin{equation*}
\text{\rsfsten L}\left \{ f_{\,1}^{\,\nu }\left ( t \right );s \right \}= \frac{\lambda }{s^{\,\nu }+\lambda}
\end{equation*}
it follows that the density of the waiting time of the $k$-th event, $ T_{k}= \sum_{j= 1}^{k}\text{\rsfsten U}_{j} $, possesses the Laplace transform
\begin{equation*}
\text{\rsfsten L}\left \{ f_{\,k}^{\,\nu }\left ( t \right );s \right \}= \frac{\lambda^{k} }{\left (s^{\,\nu }+\lambda\right )^{k}}.
\end{equation*}
Its inverse can be obtained by applying formula (2.5) of \cite{Prabhakar1971}, i.e.
\begin{equation}\label{InverseLaplace}
\text{\rsfsten L}\left \{ t^{\gamma -1}E_{\,\beta ,\gamma }^{\,\delta  }\left ( \omega t^{\beta } \right );s \right \}= \frac{s^{\,\beta \delta -\gamma }}{\left ( s^{\,\beta }-\omega  \right )^{\delta }},
\end{equation}
(where $ Re(\beta)>0,\,Re(\gamma)>0,\, Re(\delta)>0$ and $ s> \left | \omega  \right |^{\frac{1}{ Re(\beta)}} $). By setting $ \beta=\nu,\, \gamma=k\nu,\, \delta=k $ and $ \omega=-\lambda $ we have
\begin{equation}\label{densitatempoarrivo}  
f_{\,k}^{\,\nu }\left ( t \right )=\mathbb{P}\left \{ T_{k}\in d\,t \right \}/d\,t= \lambda^{k}t^{\,k\nu-1}E_{\,\nu ,k\nu  }^{\,k}(-\lambda t^{\nu }),
\end{equation}
\\
where
\begin{equation}\label{GML} %GML
E_{\alpha ,\beta }^{\gamma }(z)=\sum_{r=0}^{\infty}\frac{(\gamma )_{r}\, z^{r}}{r!\,\Gamma(\alpha r+\beta )}, \quad \alpha,\beta,\gamma \in\mathbb{C},\;Re(\alpha),Re(\beta),Re(\gamma)>0
\end{equation}
is a generalized Mittag-Leffler function and, as usual,
$(\gamma)_{r}=\gamma(\gamma+1)\dots(\gamma+r-1)$, $r=1,2,\dots,$
$(\gamma)_{0}=1$, is the Pochhammer symbol.
\\\\
The corresponding distribution function can be obtained by integrating  (\ref{densitatempoarrivo}), thus obtaining 
(see Eq.\ (2.20) of \cite{BeghinOrsingher2010})
\begin{equation}\label{PDFTempoArrivo} %come ottengo la funzione di distribuzione
F_{\,k}^{\,\nu}(t) =\mathbb{P}\left \{ T_{k}<t \right \}\\ =\lambda^{k}t^{k\nu}E_{\nu ,k\nu +1 }^{k}(-\lambda t^{\nu }).
\end{equation}
%
%It is interesting to point out that the density (\ref{InterArrivalDensity}) of the interarrival times $ \text{\rsfsten U}_{j} $ is completely monotone, and therefore log-convex. In fact, the function $ t^{\beta -1}E_{\alpha ,\beta }^{\gamma }(-\lambda t^{\nu}) $, for all $ \lambda>0 $, is completely monotone if and only if $ 0< \alpha ,\beta \leq 1 $ and $ 0<\gamma \leq \beta/\alpha $ (cf. Chapter 5 of Gorenflo {\em et al.}\ \cite{Gorenflo2014}). 
%We recall that random variables with log-convex densities are also said to have the \textit{decreasing likelihood ratio} (DLR) property.
Taking into account (\ref{PDFTempoArrivo}), the probability mass function of the process $ N_{\,\lambda}^{\,\nu }(t) $ can be easily computed as follows (see, also, Eq.\ (2.21) of \cite{BeghinOrsingher2010}):
\begin{equation}\label{PFpmf}
\mathbb{P}\left \{ N_{\,\lambda}^{\,\nu }(t)= n \right \}=\mathbb{P}\left ( T_{n}\leq t< T_{n+1} \right)=\left (\lambda t^{\nu }  \right )^{n}E_{\nu,n\nu +1 }^{n+1 }(-\lambda t^{\nu }).
\end{equation}
Moreover, recalling Eq.\ (2.29) of \cite{BeghinOrsingher2010}, we have that 
the moment generating function of the process $ N_{\,\lambda}^{\,\nu }(t) $, $t\geq 0$, can be expressed as 
\begin{equation}\label{MgfPF}
\mathbb{E}\left[e^{s N_{\,\lambda}^{\,\nu }(t)} \right]
 = E_{\nu ,1}\left ( \lambda\left ( e^{s}-1 \right )t^{\,\nu } \right ),
 \qquad s\in\mathbb{R}. 
\end{equation}
The mean and the variance of $ N_{\,\lambda}^{\,\nu }(t) $ read (see Eqs.\ (2.7) and (2.8) of \cite{BeghinOrsingher2009})
\begin{equation}\label{MeanVarPoissonFrac}
\mathbb{E}\left [ N_{\,\lambda}^{\,\nu }(t) \right ]=\frac{\lambda t^{\nu }}{\Gamma \left ( \nu +1 \right )}, 
\qquad 
\mathrm{Var}\left [ N_{\,\lambda}^{\,\nu }(t) \right ]=\frac{2\left ( \lambda t^{\nu } \right )^{2}}{\Gamma \left ( 2\nu +1 \right )}-\frac{\left (\lambda t^{\nu }  \right )^{2}}{\left ( \Gamma \left ( \nu +1 \right ) \right )^{2}}+\frac{\lambda t^{\nu }}{\Gamma \left ( \nu +1 \right )}.
\end{equation}
In general, the analytical expression for the \textit{m}th order moment of the fractional Poisson process is given by (cf. \cite{Laskin2009}, Eq.\ (40))
\begin{equation}\label{moments}
 \mathbb{E}\left [ \left ( N_{\,\lambda}^{\,\nu }(t) \right)^m\right ]
= \sum_{l=0}^{m}\textit{S}_{\nu }\left ( m,l \right )\left ( \lambda t^{\nu} \right )^{l},
\end{equation}
where $ \textit{S}_{\nu }\left ( m,l \right ) $ is the fractional Stirling number defined by Eq.\ (32) of \cite{Laskin2009}.

%%%%%%%%%%%%%%%%%%%
\section{Fractional counting process}\label{Section3}
%%%%%%%%%%%%%%%%%%%
Let $\{M^1(t); t\geq 0\}$ be a counting process defined by  following rules:
\begin{enumerate}
\item $ M^1(0)=0\;\; $ a.s.;
\item $M^1(t)$ has stationary and independent increments;
\item $ \mathbb{P}\{ M^1(h)= j \}= \lambda _{j}h+\text{o}  ( h ) $, \ for $j=1,2,\ldots, k$; 
\item $ \mathbb{P}\{M^1(h)>k \}= \text{o} (h)$, 
\end{enumerate}
where $k\in\mathbb{N}\equiv\{1,2,\ldots\}$ is fixed, and $ \lambda_{1} ,\lambda_{2}, \ldots, \lambda_{k}>0$. 
From the above assumptions we have that the probability distribution 
$p_{\, j}(t)=\mathbb{P}\left \{ M^1\left ( t \right )= j \right \}$, 
for $j\in \mathbb{N}_0\equiv\{0,1,2,\ldots\}$, satisfies the following system of difference-differential equations:
\begin{equation}\label{SysteqM1}
\dfrac{\mathrm{d}p_{j}(t) }{\mathrm{d}t}
=\sum_{r=1}^{k} \lambda_{r}\,p_{\,j-r}(t)-(\lambda_{1}+\ldots+\lambda_{k})\,p_{j}(t),\qquad  t>0,
\end{equation}
where $p_{\,j}(t)=0$ for $j<0$. 
\\
\par
In this section we examine a fractional extension of $\{M^1(t); t\geq 0\}$. 
We obtain a proper probability distribution and explore the main properties of the corresponding fractional process. 
%%%%%%%%%%%%%%%%%%
\subsection{The probability law}
%%%%%%%%%%%%%%%%%%%
With reference to the fractional derivatives  
\begin{equation*}
\frac{\mathrm{d}^{\nu }f\left ( t \right ) }{\mathrm{d} t^{\nu }}=
\begin{cases}
\frac{1}{\Gamma \left ( 1-\nu  \right )}\int_{0}^{t}\frac{\left ( \mathrm{d}/\mathrm{d}s \right )f\left ( s \right )}{\left ( t-s \right )^{\nu }}\,\mathrm{d}s,& 0< \nu < 1,\\[0.3cm] 
f{}\,'\left ( t \right )&\nu=1,
\end{cases}
\end{equation*}
let us now introduce a fractional extension of the process $M^1(t)$. For all fixed $\nu\in (0,1]$ and 
$k\in\mathbb{N}$, let $\{M^{\nu}(t); t\geq 0\}$ be a counting process, and assume that the probability distribution  
\begin{equation}\label{FracPoisson2}
p_{\, j}^{\, \nu }(t)=\mathbb{P}\left \{ M^{\nu}\left ( t \right )= j \right \}, 
\qquad  j\in \mathbb{N}_0 
\end{equation}
satisfies the following system of fractional difference-differential equations 
\begin{equation}  %equazioni Poisson frazionario piðËü tassi con derivate
\label{Pfraz}
\begin{cases}
\dfrac{\mathrm{d}p_{\,0}^{\, \nu }(t) }{\mathrm{d}t^{\nu} }=-\Lambda\, p_{\,0}^{\, \nu }(t)
\\[0.3cm] 
\dfrac{\mathrm{d}p_{\,j}^{\, \nu }(t) }{\mathrm{d}t^{\nu} }
=\sum_{r=1}^{j} \lambda_{r}\,p_{\,j-r}^{\, \nu }(t)-\Lambda\,p_{\,j}^{\, \nu }(t),& \quad \text{$ j=1,2,\dots,k-1 $}
\\[0.3cm] 
\dfrac{\mathrm{d}p_{\,j}^{\, \nu }(t) }{\mathrm{d}t^{\nu} }
=\sum_{r=1}^{k} \lambda_{r}\,p_{\,j-r}^{\, \nu }(t)-\Lambda\,p_{\,j}^{\, \nu }(t),& \quad \text{$ j=k,k+1,\dots $},
\end{cases}
\end{equation}
for $\Lambda= \lambda_{1}+\lambda_{2}+ \ldots+\lambda_{k}$, together with the condition
\begin{equation}\label{InitialConditionsPoissonFrac} % condizioni iniziali
p_{\, j}(0)=
\begin{cases}
1, & \text{$ j=0 $} \\
0, & \text{$ j\geq 1 $.}
\end{cases}
\end{equation}
Clearly, when $\nu=1$ the system (\ref{Pfraz}) identifies with the  difference-differential equations 
of process $M^1(t)$ given in (\ref{SysteqM1}). Furthermore, when $k=1$ the process $M^{\nu}(t)$ 
identifies with the process $N_{\,\lambda}^{\,\nu }(t)$ considered in Section \ref{Section1}. 
\par
Hereafter we will obtain the solution to (\ref{Pfraz})-(\ref{InitialConditionsPoissonFrac}) in 
terms of the generalized Mittag-Leffler function (\ref{GML}) and show that it represents a true probability 
distribution of $ M^{\nu}(t)$. To this purpose we first obtain the moment generating function of $M^{\nu}(t)$ 
in terms of the Mittag-Leffler function. 
\begin{proposition}\label{fgmfraz}
For all fixed $\nu\in (0,1]$ and $k\in\mathbb{N}$, the moment generating function of $M^{\nu}(t)$ is given by 
\begin{equation}\label{fgmomM}
 \mathbb{E}\left[e^{s M^{\nu}(t)} \right]
 = E_{\,\nu,1}\Big( \sum_{j=1}^k \lambda _{j} \left( e^{j s}-1 \right )\,t^{\,\nu } \Big ),
 \qquad t\geq 0, \;\; s\in\mathbb{R}.
\end{equation}
\end{proposition}
\begin{proof}
From system (\ref{Pfraz}) and condition (\ref{InitialConditionsPoissonFrac}) we have that the 
probability generating function $G(z,t):= \mathbb{E}\left[z^{M^{\nu}(t)} \right]$ satisfies the Cauchy problem
\begin{equation*}   
\begin{cases}
\displaystyle\frac{\partial G(z,t)}{\partial t^{\nu}}
=- \sum_{j=1}^k \lambda _{j} \left(1-z^{j} \right )\,G(z,t)
\\[0.3cm] 
G(z,0)=1.
\end{cases}
\end{equation*}
By adopting a Laplace-transform approach we obtain 
\begin{equation*}
 \text{\rsfsten L}\left \{ G(z,t);s \right \}
 = \frac{s^{\nu-1} }{s^{\nu}+ \sum_{j=1}^k \lambda_j(1-z^j) }.
\end{equation*}
Eq.\ (\ref{fgmomM}) thus follows recalling Eq.\ (\ref{InverseLaplace}). 
\end{proof}
\par
%Since $ E_{\,\nu ,1}(0)=1 $, formula (\ref{pgfPoissonFrac}) is useful in checking that $ p_{\,j}^{\,\nu}\left ( t \right ) $ represents a true probability distribution.

We remark that the use of the Caputo fractional derivative permits us to avoid fractional initial conditions in the 
previous proof since, in general,
\begin{equation*}
\text{\rsfsten L}\left\{f^{\nu};s\right\}=s^{\nu}\text{\rsfsten L}\left\{f;s\right\}-s^{\nu-1}f\bigg|_{x=0},\qquad\nu\in(0,1].
\end{equation*}
\par
Let us now show that $M^{\nu}(t)$ can be expressed as a compound fractional Poisson process. 
\begin{proposition}\label{CFPRemark} 
For all fixed $\nu\in (0,1]$ we have 
\begin{equation}\label{CompoundProcess}
M^{\nu}(t)\overset{d}{=}\sum_{i=1}^{N_{\,\Lambda}^{\,\nu }(t)}X_{i},\qquad t\geq 0,
\end{equation}
where $ N_{\,\Lambda}^{\,\nu }(t) $ is a fractional Poisson process, defined as in (\ref{PoissonFractional}), with intensity $ \Lambda= \lambda_{1}+\lambda_{2}+ \ldots+\lambda_{k}$. 
Moreover, $ \left \{ X_{n}:n\geq 1 \right \} $ is a sequence of i.i.d.\ random variables, independent of 
$ N^{\,\nu}_{\,\Lambda}(t)$, such that for any $ n\in\mathbb{N}$
\begin{equation}\label{Jumps}
\mathbb{P}\{X_{n}=j\}=\frac{\lambda _{j}}{\Lambda}, 
\qquad j=1,2,\ldots,k 
\end{equation}
and where both $ N^{\,\nu}_{\,\Lambda}(t) $ and $ X_{n} $ depend on the same parameters 
$ \lambda_{1},\lambda_{2}, \ldots,\lambda_{k}$.
\end{proposition}
\begin{proof}
The moment generating function of $Y(t):=\sum_{i=1}^{N_{\,\Lambda}^{\,\nu }(t)}X_{i}$, $t\geq 0$, 
can be expressed as 
\begin{align}
\mathbb{E}\left [ e^{s Y(t)} \right ]
& = \mathbb{E}\left[\mathbb{E}\left [ e^{s Y(t)} \Big| N_{\,\Lambda}^{\,\nu }(t)\right ]\right]
\notag\\
&= \mathbb{E}\left[\left(\mathbb{E}\left [e^{s X_1}\right]\right)^{N_{\,\Lambda}^{\,\nu }(t)}\right].
\notag
\end{align} 
Hence, since 
$$
 \mathbb{E}\left [e^{s X_1}\right]=\frac{1}{\Lambda}\,\sum_{j=1}^k \lambda_j \,e^{j s},
$$
we have 
$$
\mathbb{E}\left [ e^{s Y(t)} \right ]
=\mathbb{E}\left [e^{N_{\,\Lambda}^{\,\nu }(t)\ln \left(\frac{1}{\Lambda}\,\sum_{j=1}^k \lambda_j \,e^{j s}\right)}\right ].
$$
Finally, making use of Eq.\  (\ref{MgfPF}) we immediately obtain that the 
moment generating function of $Y(t)$ identifies with the right-hand-side of (\ref{fgmomM}). 
This completes the proof. 
\end{proof}
\par
We remark that, due to Proposition \ref{CFPRemark}, $ M^{\nu}(t) $ can be regarded as a special case of the process defined in Eq.\ (7) of \cite{BeghinMacci2014}, under a suitable choice of the probability mass function $ \left ( q_{k}\right )_{k\geq 1} $ and the parameter $ \lambda $. Furthermore, 
according to Definition 7.1.1 of \cite{BeningKorolev}, the process $ M^{\nu}(t) $ is a compound Cox process, since \cite{BeghinOrsingher2010} show that $ N_{\,\Lambda }^{\,\nu }\left ( t \right ) $ is a Cox process with a proper directing measure. Moreover, $ M^{\nu}(t) $ is a compound fractional process, and thus it is neither Markovian nor L$ \grave{\mathrm{e}} $vy (cf. \cite{Scalas2011}).
\par
We are now able to obtain the probability mass function (\ref{FracPoisson2}) of $ M^{\nu}(t) $. Indeed, the following Proposition holds true.
\begin{proposition}\label{prop:pmf}
The solution $ p_{\, j}^{\, \nu }(t)$ of the Cauchy problem (\ref{Pfraz})-(\ref{InitialConditionsPoissonFrac}), 
for $ j\in \mathbb{N}_0$, $\nu\in (0,1]$ and $ t\geq 0 $, is given by
\begin{equation}\label{pmfPoissonFrac}
p_{\, j}^{\, \nu }(t)
=\sum_{r=0}^{j}\, \sum_{\substack{\alpha_{1}+\alpha_{2}+\ldots+\alpha_{k}=r\\\alpha_{1}+2\alpha_{2}+\ldots+k\alpha_{k} =j}}\binom{r}{\alpha_{1},\alpha_{2},\ldots,\alpha_{k}}\, \lambda _{1}^{\alpha_{1}}\lambda _{2}^{\alpha _{2}}\ldots\lambda _{k}^{\alpha _{k}}\, t^{r\nu}E_{\nu,r\nu +1 }^{r+1 }(-\Lambda t^{\nu }).
\end{equation}
\end{proposition}
\begin{proof}
From (\ref{CompoundProcess}) and from a conditioning argument we have
$$
p_{\, j}^{\, \nu }(t)=\mathbb{P}\left \{ M^{\nu}\left ( t \right )= j \right \}= \sum_{r=0}^{j}\mathbb{P}\left \{ X_{1}+X_{2}+\ldots+X_{r}=j \right \}\mathbb{P}\left \{ N_{\,\Lambda}^{\,\nu }(t)= r \right \}.
$$
Since $ X_{1},X_{2},\ldots, X_{r} $ are independent and identically distributed (cf.\ (\ref{Jumps})), it follows that
\begin{align*}
 \mathbb{P}\left \{ X_{1}+X_{2}+\ldots+X_{r}=j \right \}  
 =&  \sum_{\substack{\alpha_{1}+\alpha_{2}+\ldots+\alpha_{k}=r\\\alpha_{1}+2\alpha_{2}+\ldots+k\alpha_{k} =j}}\binom{r}{\alpha_{1},\alpha_{2},\ldots,\alpha_{k}} \\
 &\times  \left ( \frac{\lambda _{1}}{\Lambda } \right )^{\alpha_{1}}\left ( \frac{\lambda_{2}}{\Lambda } \right )^{\alpha _{2}}\ldots \left ( \frac{\lambda _{k}}{\Lambda } \right )^{\alpha_{k}},
\end{align*}
where the sum is taken in order to consider all the possible ways of having $ r $ jumps, with 
$ \alpha_{1} $ jumps of size 1, $\ldots$, $\alpha_{k} $ jumps of size $ k $, and such that the total amplitude, i.e.\ 
$ \alpha_{1}+2\alpha_{2}+\ldots+k\alpha_{k} $, equals $ j $. 
Hence, recalling formula (\ref{PFpmf}), the Proposition follows.
\end{proof}
\par
Proposition \ref{prop:pmf} is an extension of Proposition 2 of \cite{DiCrMaMe}, which is concerning 
with case $k=2$. 
Some plots of probabilities (\ref{pmfPoissonFrac}) are shown in Figure \ref{fig:prob} and Figure \ref{fig:prob2}.
\par
From (\ref{pmfPoissonFrac}) we note that, for $\nu\in (0,1]$,  
$$
 p_{\,0}^{\, \nu }(t)=E_{\nu,1 }(-\Lambda t^{\nu }), \qquad t\geq 0. 
$$
Moreover, making use of  Eqs.\ (\ref{GML}) and (\ref{pmfPoissonFrac}) we obtain hereafter the 
distribution of the process $ M^{\nu}(t) $ in the special case $\nu=1$. 
\begin{corollary}
The  probability mass function  $ p_{\, j}^{\,1}(t)$, for $ j\in \mathbb{N}_0$ and $ t\geq 0 $, is given by 
\begin{equation}\label{pmfPoissonFrac1}
p_{\, j}^{\, 1}(t)
=\sum_{r=0}^{j}\, 
\sum_{\substack{\alpha_{1}+\alpha_{2}+\ldots+\alpha_{k}=r\\\alpha_{1}+2\alpha_{2}+\ldots+k\alpha_{k} =j}}
\frac{\lambda _{1}^{\alpha_{1}}\lambda _{2}^{\alpha _{2}}\ldots\lambda _{k}^{\alpha _{k}}}
{\alpha_{1}!\,\alpha_{2}!\,\ldots\,\alpha_{k}!}\,  t^{r} e^{-\Lambda t}.
\end{equation}
\end{corollary}
%
%--------------------------------------------------------------------
\begin{figure}[t]  %%%%%%%%%%% FIGURA  PROB
%\vspace{20mm}
\begin{center}
\epsfxsize=6.2cm
\epsfbox{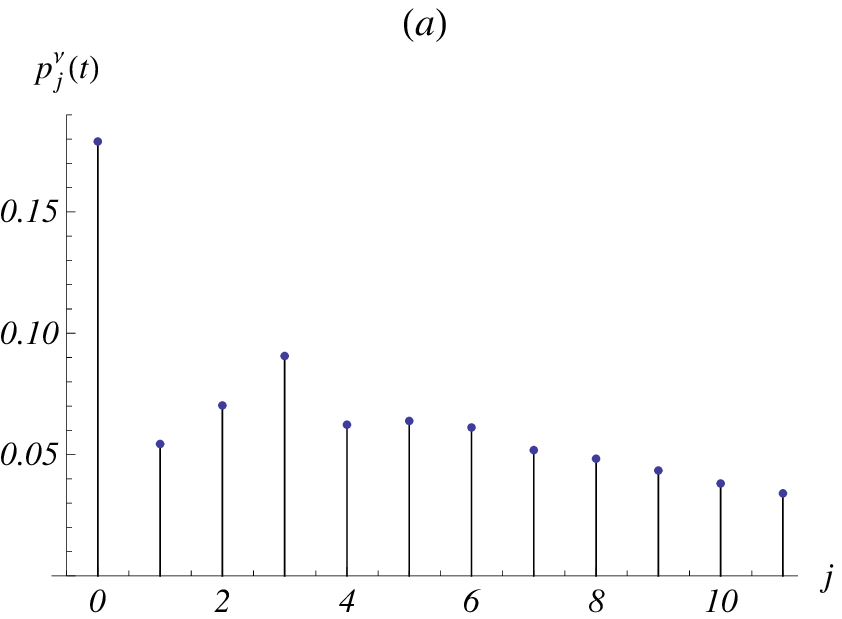}
$\;$
\epsfxsize=6.2cm
\epsfbox{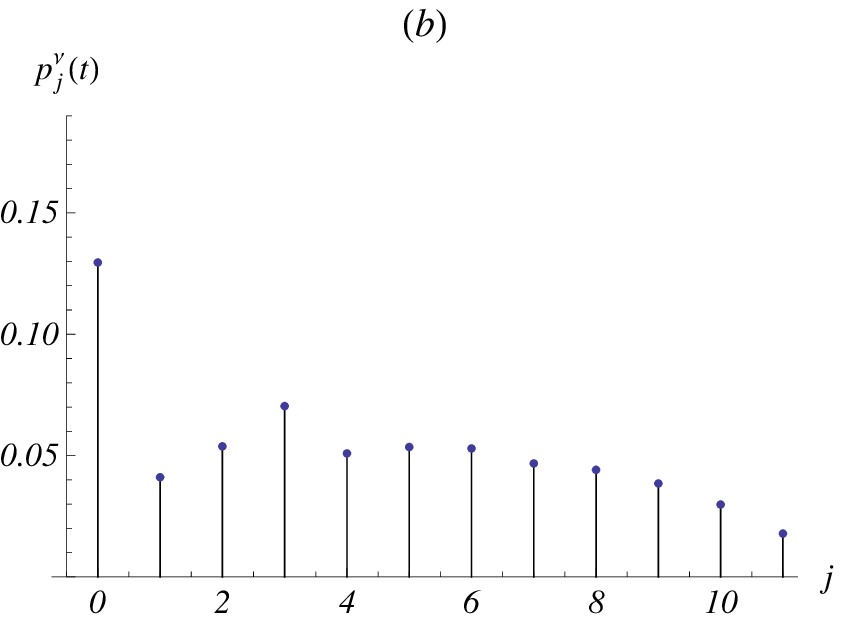}
\end{center}
\vspace{-0.2cm}
\caption{Probability distribution of  $M^{\nu}(t)$, given in (\ref{pmfPoissonFrac}),
for $j=0,1,\ldots, 11$, with $k=3$, $\nu=0.5$, $\lambda_1=\lambda_2=\lambda_3=1$, 
$(a)$ $t=1$ and $(b)$ $t=2$.
The displayed probability mass is $(a)$ 0.797292 and $(b)$ 0.629278.}
\label{fig:prob}
\end{figure}
%--------------------------------------------------------------------  
%--------------------------------------------------------------------
\begin{figure}[t]  %%%%%%%%%%% FIGURA  PROB t
%\vspace{20mm}
\begin{center}
\epsfxsize=6.2cm
\epsfbox{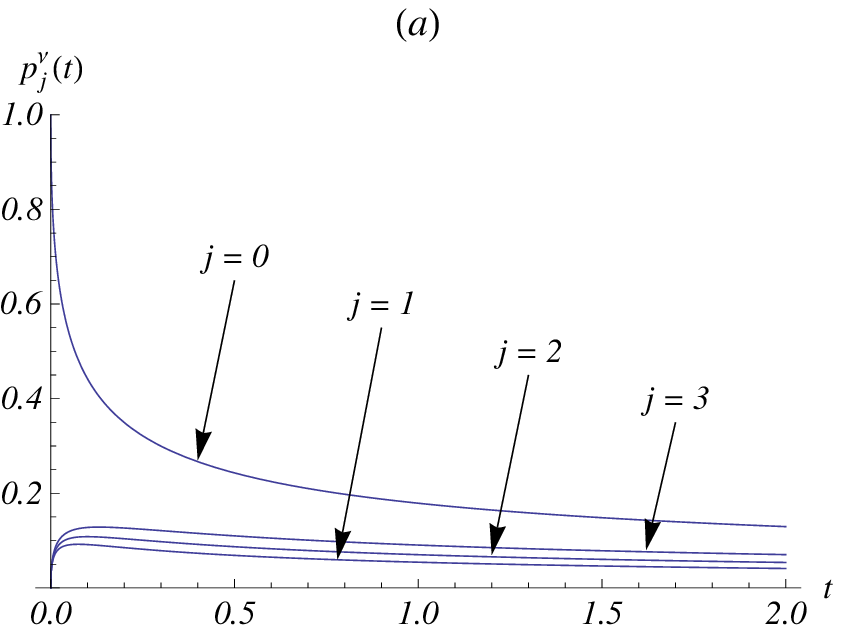}
$\;$
\epsfxsize=6.2cm
\epsfbox{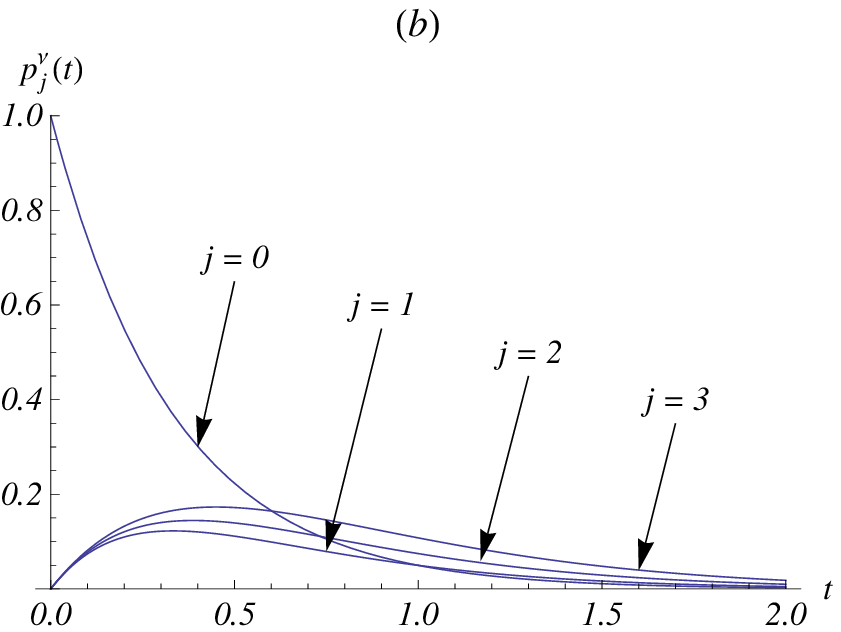}
\end{center}
\vspace{-0.2cm}
\caption{Probability distribution of  $M^{\nu}(t)$, given in (\ref{pmfPoissonFrac}),
for $0\leq t\leq 2$, with $k=3$, $\lambda_1=\lambda_2=\lambda_3=1$, 
$(a)$ $\nu=0.5$ and $(b)$ $\nu=1$.}
\label{fig:prob2}
\end{figure}
%--------------------------------------------------------------------  
%%%%%%%%%%%%%%%%%%%%
\subsection{Equivalent representation}
%%%%%%%%%%%%%%%%%%%%
We will now examine an interesting relationship between the process $ M^{\nu}(t) $ and the 
process $M^{1}(t)$. In fact, we  show that the following representation holds:
\begin{equation*}
M^{\nu}(t)\overset{d}{=}M^{1}\left (\text{\rsfsten T} _{2\nu }\left ( t \right )\right ),
\end{equation*}
where $ \text{\rsfsten T} _{2\nu }\left ( t \right ) $ is a suitable random process, and thus $ M^{\nu}(t) $ can be considered as a homogeneous Poisson-type counting process with jumps of sizes $1,2,\ldots,k$ stopped at a random time 
$ \text{\rsfsten T} _{2\nu }\left ( t \right ) $.
\par
Let us denote by $ g(z,t)=g_{\,2\nu }\left ( z,t \right ) $ the solution of the Cauchy problem 
\begin{equation}\label{Subordinator}
\begin{cases}
\frac{\partial^{\,2\nu } g(z,t) }{\partial t^{\,2\nu }}=\frac{\partial^{\,2} g(z,t)}{\partial z^{\,2}},&t>0,\;z\in\mathbb{R}\\
g \left ( z,0 \right )= \delta \left ( z \right ),&0< \nu < 1\\
\frac{\partial g\left ( z,t \right )}{\partial t}\bigg|_{t=0}= 0,&\frac{1}{2}< \nu < 1.
\end{cases}
\end{equation}
It is well-known that (see \cite{Mainardi1996} and \cite{MainardiAltro1996})
\begin{equation}\label{density}
g_{\,2\nu }\left ( z,t \right )=\frac{1}{2t^{\nu }}W_{-\nu ,1-\nu }\left ( -\frac{|z|}{t^{\nu }} \right ), \quad t>0,\;z\in\mathbb{R},
\end{equation}
where
\begin{equation}\label{WrightFunction}
W_{\alpha,\beta }\left ( x \right )= \sum_{k=0}^{\infty}\frac{x^{\,k}}{k!\,\Gamma \left (\alpha k+\beta   \right )},\qquad\alpha > -1,\,\beta > 0,\,x\in\mathbb{R}
\end{equation}
is the Wright function. Let 
\begin{equation}\label{FoldedSolution}
\bar{g}_{\,2\nu }\left ( z,t \right )=
\begin{cases}
2\,g_{\,2\nu }\left ( z,t \right ),  &z>0\\
0,  &z<0
\end{cases}
\end{equation}
be the folded solution to (\ref{Subordinator}) and let $ \text{\rsfsten T} _{2\nu }\left ( t \right ) $ be a random process (independent from the process $ M^{1}\left ( t \right ) $) whose transition density $ \mathbb{P}\left \{\text{\rsfsten T} _{2\nu }\left ( t \right ) \in dz \right \}/dz $ is given in %(\ref{density}) and 
(\ref{FoldedSolution}).
\begin{remark}\rm
It has been proved in \cite{Orsingher2004} that the solution $ g_{2\nu} $ to (\ref{Subordinator}) can be alternatively expressed as 
\begin{equation*}
g_{2\nu }\left ( z,t \right )= \frac{1}{2\Gamma \left ( 1-\nu  \right )}\int_{0}^{t}\left ( t-w \right )^{-\nu }f_{\nu }\left ( w,\left | z \right | \right )dw,\quad z\in \mathbb{R},
\end{equation*}
where $ f_{\nu}\left(\cdot, y\right) $ is a stable law $ S_{\nu}\left(\mu,\beta,\sigma\right) $ of order $ \nu $, with parameters $ \mu=0, \beta=1 $ and $ \sigma =\left(z\cos \frac{\pi\nu}{2}\right)^{\frac{1}{\nu}} $.
\end{remark}
\begin{proposition}\label{pr:indent}
The process $ M^{\nu}\left ( t \right ) $ and the process $ M^{1}\left (\text{\rsfsten T} _{2\nu }\left ( t \right )\right ) $ are identically distributed.
\end{proposition}
\begin{proof}[Proof]
From (\ref{FracPoisson2}) and (\ref{FoldedSolution}) we have 
$$
 \mathbb{P}\left\{ M^{1} \left (\text{\rsfsten T} _{2\nu }\left ( t \right )\right )= n \right \}
 =\int_{0}^{\infty} p_n^1 (z) \,\bar{g  }_{\,2\nu }\left ( z,t \right )dz.
$$
Hence, making use of (\ref{pmfPoissonFrac1}) and (\ref{WrightFunction}) we get
\begin{align*}
\mathbb{P}\left\{ M^{1} \left (\text{\rsfsten T} _{2\nu }\left ( t \right )\right )= n \right \}
 =& \sum_{j=0}^{n}\, 
\sum_{\substack{\alpha_{1}+\alpha_{2}+\ldots+\alpha_{k}=j\\\alpha_{1}+2\alpha_{2}+\ldots+k\alpha_{k} =n}}
\frac{\lambda _{1}^{\alpha_{1}}\lambda _{2}^{\alpha _{2}}\ldots\lambda _{k}^{\alpha _{k}}}
{\alpha_{1}!\,\alpha_{2}!\,\ldots\,\alpha_{k}!} \\  
&\times   \frac{1}{t^{\nu}}\,
 \int_{0}^{\infty}e^{-\Lambda z}\,z^{j}\,W_{-\nu ,1-\nu }\left ( -\frac{z}{t^{\nu}} \right )dz.
\end{align*}
For $y=\Lambda z$, the last expression identifies with (\ref{pmfPoissonFrac}) due to the following integral representation of the generalized Mittag-Leffler function in terms of the Wright function, proposed by \cite{BeghinOrsingher2010}:
\begin{equation*}
E_{\nu,k\nu +1 }^{k+1 }(-\Lambda t^{\nu })=\frac{1}{k!\,\Lambda^{k+1}\,t^{\left ( k+1 \right )\nu }}\int_{0}^{\infty}e^{-y}\,y^{\,k}\,W_{-\nu ,1-\nu }\left ( -\frac{y}{\Lambda t^{\nu}} \right )dy.
\end{equation*}
This completes the proof.
%\qedhere
%
\end{proof}
\begin{remark}\label{rm:new}\rm 
Since the transition density (\ref{FoldedSolution}) coincides with the 
probability density function of the standard inverse $\nu$-stable subordinator ${\tt E}^{\,\nu}(t)$
(see \cite{Meerschaert2011}),  the result given in Proposition \ref{pr:indent} can  
be stated also as follows: The process $ M^{\nu}\left ( t \right ) $ and the process 
$ M^{1}\left ({\tt E}^{\,\nu}(t)\right ) $ are identically distributed. 
\end{remark}
\begin{remark}\label{rm:FracPoisson}\rm
In \cite{BeghinOrsingher2010} Beghin and Orsingher proved an analogous subordination relationship, i.e.
\begin{equation*}
 N_{\,\lambda}^{\,\nu }(t)\overset{d}{=}N_{\,\lambda}^{\,1 }(\text{\rsfsten T} _{2\nu }\left ( t \right )),
\end{equation*}
where $  N_{\,\lambda}^{\,\nu }(t) $ is the fractional Poisson process defined in (\ref{PoissonFractional}) and $ \text{\rsfsten T} _{2\nu }\left ( t \right ) $ is the random time defined above.
\end{remark}
\begin{remark}\rm
By taking $ \nu=\frac{1}{2} $, from Proposition \ref{pr:indent} we have that  $ M^{1/2}\left ( t \right ) $ and  
$ M^{1}\left (\text{\rsfsten T} _{1}\left ( t \right )\right ) $ are identically distributed.
We note that the random time $ \text{\rsfsten T} _{1}\left ( t \right )$, $t>0 $, becomes a reflecting Brownian motion. Indeed, in this case equation (\ref{Subordinator}) reduces to the heat equation 
\begin{equation*}
\begin{cases}
\frac{\partial g }{\partial t}=\frac{\partial^{\,2} g}{\partial z^{\,2}},&t>0,\;z\in\mathbb{R}\\
g \left ( z,0 \right )= \delta \left ( z \right ),
\end{cases}
\end{equation*}
and the solution $ g_{1}\left(z,t\right) $ is the density of a Brownian motion $ B\left(t\right),t>0 $, with infinitesimal variance 2. After folding up the solution, we find the following probability mass 
\begin{equation*}
\begin{split}
\mathbb{P}\left \{ M^1\left (\text{\rsfsten T} _{1}\left ( t \right )\right )=n \right \}
&=\int_{0}^{\infty}p_n^1(z)
\,\frac{e^{-\frac{z^{2}}{4t}} }{\sqrt{\pi} t}dz
\\
&=\mathbb{P}\left \{ M^1\left(\left | B\left ( t \right ) \right |\right)= n \right \},
\end{split}
\end{equation*}
so that $M^{1/2}\left(t\right) $ is a jump process at a Brownian time. 
\end{remark}
\begin{remark}\rm
It is worth noticing that both the compositions of the fractional Poisson process  $ N_{\,\lambda}^{\,\nu }(t) $ defined in (\ref{PoissonFractional}) and of the fractional process $ M^{\nu}(t) $ defined in (\ref{FracPoisson2}) with the random time $ \text{\rsfsten T} _{2\nu }\left ( t \right ) $ yields again fractional processes of different order, i.e.
\begin{equation*}
N_{\,\lambda}^{\,\nu }(\text{\rsfsten T} _{2\nu }\left ( t \right ))\overset{d}{=}N_{\,\lambda}^{\,\nu^{2} }(t)
\qquad 
\hbox{and}
\qquad
 M^{\nu}(\text{\rsfsten T} _{2\nu }\left ( t \right ))\overset{d}{=}M^{\nu^{2}}(t).
\end{equation*}
Taking into account the subordinating relationships examined in Proposition \ref{pr:indent} and in 
Remark \ref{rm:FracPoisson}, this fact follows immediately from Remark 3.1 of \cite{Kumar2011}, 
since, in general, the composition of two stable subordinators of indexes $ \beta_{1} $ and $ \beta_{2} $ 
respectively is a stable subordinator of index $ \beta_{1}\beta_{2} $.
\end{remark}
\begin{remark}\rm
\-\hspace{8pt}Bearing in mind Proposition \ref{CFPRemark},  setting 
$$
 {\mathcal S}_r=\Lambda \cdot \mathbb{E}[X^{r}]=\sum_{j=1}^k j^r\, \lambda_j, \qquad r=1,2
$$
and recalling (\ref{MeanVarPoissonFrac}), we can compute more effortlessly 
the mean and the variance of the process. In fact, by Wald's equation we have
\begin{align*}
\mathbb{E}\left [ M^{\nu}(t) \right ]& = \mathbb{E}[X]\cdot \mathbb{E}\left [ N_{\,\Lambda}^{\,\nu } \left ( t \right )\right ]\notag\\
&=\frac{{\mathcal S}_1\,t^{\nu}}{\Gamma \left ( \nu+1 \right )},\qquad t\geq 0.
\end{align*} 
Moreover, by the law of total variance we get   
\begin{align*}
\mathrm{Var}\left [ M^{\nu}(t) \right ]&= \mathrm{Var}\left [ X \right ]\cdot \mathbb{E}\left [ N_{\,\Lambda }^{\,\nu} (t)\right ]
+\left ( \mathbb{E}\left [ X \right ] \right )^{2}\cdot \mathrm{Var}\left [  N_{\,\Lambda }^{\,\nu} (t)\right ]
\notag\\
&=\frac{{\mathcal S}_2\,t^{\nu}}{\Gamma ( \nu+1)}
+ {\mathcal S}_{1}^2 \,t^{2\nu}\,Z(\nu), \qquad t\geq 0,
\notag
\end{align*}
where 
\begin{equation*}
Z(\nu):=\frac{1}{\nu }\left ( \frac{1}{\Gamma \left ( 2\nu \right )}-\frac{1}{\nu \Gamma ^{2}(\nu)} \right ).
\end{equation*}
As a consequence it is not hard to show that $ \mathrm{Var}\left [ M^{\nu}(t) \right ]-\mathbb{E}\left [ M^{\nu}(t) \right ]>0 $, or, equivalently, that the process $ M^{\nu}(t) $ exhibits overdispersion, since $ Z(\nu)>0 $ for all $ \nu\in (0,1) $ and $ Z(1)=0 $. Finally, we point out that a formal expression for the moments of process $ M^{\nu}(t)$ 
is provided in Lemma \ref{lemma}. 
\end{remark}
%
%%%%%%%%%%%%%%%%%%%%%%%%%%%%%%
\section{Waiting times and first-passage times}\label{Section4}
%%%%%%%%%%%%%%%%%%%%%%%%%%%%%%
We evaluate the probability distribution function of the waiting time until the first occurrence of a jump of size $ i,\,i=1,2,\ldots, k $, for the process $ M^{\nu}(t) $. We first observe that the following decomposition holds:
$$
 M^{\nu}(t)=\sum_{j=1}^k j\,M^{\nu}_{j}(t),\qquad t\geq 0,
$$
where
\begin{equation}
M^{\nu}_{j}(t):=\sum_{i= 1}^{N_{\, \Lambda}^{\, \nu }\left ( t \right )} {\bf 1}_{\left\{X_{i}=j\right\}},
\qquad j=1,2,\ldots, k,
\end{equation}
and thus $ M^{\nu}_{j}(t) $ counts the number of jumps of amplitude $j $ performed by $ M^{\nu}(t) $ in $ (0,t] $. Furthermore, we introduce the random variables 
\begin{equation*}
H_{j}:=\inf \left \{ s>0:M_{j}^{\nu}(s)=1 \right \}\quad\mathrm{and}\quad G^{j}\sim \mathrm{Geo}\left ( \frac{\lambda _{j}}{\Lambda} \right ),\qquad j=1,2,\ldots, k.
\end{equation*}
In other words, $ H_{j} $ represents the first occurrence time of a jump of amplitude $j$ for process $ M^{\nu}(t) $,
whereas $ G^{j} $ is a geometric random variable with parameter $ \frac{\lambda _{j}}{\Lambda} $ 
that describes the order of the first jump of amplitude $j$ in the sequence of jumps of $M^{\nu}(t)$. 
We prove that $ H_{j} $ is distributed as the waiting time of the first event of the fractional Poisson process defined in (\ref{PoissonFractional}) with parameter $\lambda_j$. Indeed, the following result holds.
\begin{theorem}\label{th:WaitingTime}
Let $ j\in \left \{ 1,2,\ldots,k \right \} $. Then
\begin{equation}
\mathbb{P}\left \{ H_{j}\leq t \right \}=\lambda_{j}t^{\nu}E_{\,\nu ,\nu +1}\left ( -\lambda _{j}t^{\nu} \right ),\qquad t>0.
\end{equation}
\end{theorem}
\begin{proof}[Proof]
By conditioning on $ G^{j} $, for $ t>0 $, due to Eqs.\ (\ref{CompoundProcess}) and (\ref{PDFTempoArrivo}) we have
\begin{equation*}
\begin{split}
\mathbb{P}\left \{ H_{j}\leq t \right \}&
=\mathbb{E}_{\,G^{j}}\left [ \mathbb{P}\left \{ H_{j}\leq t \mid G^{j}\right \} \right ]\\
&=\sum_{n=1}^{+\infty}\mathbb{P}\left \{ H_{j}\leq t\mid G^{j}= n \right \}\mathbb{P}\left \{  G^{j}= n\right \}\\
&=\sum_{n=1}^{+\infty}F_{\,n}^{\,\nu}(t)\, \frac{\lambda _{j}}{\Lambda}\, \left (1- \frac{\lambda _{j}}{\Lambda}\right)^{n-1}\\
&=\sum_{n=1}^{+\infty} \Lambda^{n} t^{n\nu}E_{\nu ,n\nu +1 }^{n}(-  \Lambda  t^{\nu })\,
\frac{\lambda _{j}}{\Lambda}\, \left (1- \frac{\lambda _{j}}{\Lambda}\right)^{n-1}\\
&=\lambda _{j} t^{\nu}\sum_{n=0}^{+\infty} \Lambda^{n} t^{n\nu} \left (1- \frac{\lambda _{j}}{\Lambda}\right)^{n}
E_{\nu ,(n+1)\nu +1 }^{n+1}(-  \Lambda  t^{\nu }).
\end{split}
\end{equation*}
By using formula (2.3.1) of \cite{Mathai2008}, i.e.
\begin{equation*}
\frac{1}{\Gamma \left ( \alpha  \right )}\int_{0}^{1}u^{\gamma -1}\left ( 1-u \right )^{\alpha -1}E_{\, \beta ,\gamma }^{\, \delta }\left ( zu^{\beta } \right )du= E_{\, \beta ,\gamma+\alpha }^{\, \delta }\left ( z \right ),
\end{equation*}
(where $ Re(\alpha)>0,Re(\beta)>0$ and $ Re(\gamma)>0 $) for $ \alpha=\beta=\nu$, 
$\gamma=n\nu+1$, $\delta=n+1$ and $z=-\Lambda t^{\nu} $, we get
\begin{align*}
 \mathbb{P}\left \{ H_{j}\leq t \right \}
 =&\, \frac{\lambda _{j} t^{\nu}}{\Gamma(\nu)}
 \sum_{n=0}^{+\infty} \Lambda^{n} t^{n\nu} \left (1- \frac{\lambda _{j}}{\Lambda}\right)^{n}  
 \int_{0}^{1}u^{n\nu}\left ( 1-u \right )^{\nu -1}E_{\,\nu,n\nu+1}^{\,n+1}\left ( - \Lambda\, t^{\nu}u^{\nu } \right )du
 \\
=&\, \frac{\lambda _{j} t^{\nu}}{\Gamma(\nu)}\int_{0}^{1}\left ( 1-u \right )^{\nu -1}
 \sum_{n=0}^{+\infty} \left[\Lambda t^{\nu}\left (1- \frac{\lambda _{j}}{\Lambda}\right)u^{\nu}\right]^n
E_{\,\nu,n\nu+1}^{\,n+1}\left ( - \Lambda\, t^{\nu}u^{\nu } \right )du.
\end{align*}
Due to formula (2.30) of \cite{BeghinOrsingher2010}, i.e.
$$
\sum_{n=0}^{+\infty}\left ( \lambda wt^{\nu } \right )^{n}E_{\,\nu,\nu n+1}^{\,n+1}\left ( -\lambda t^{\nu } \right )
= E_{\,\nu,1}\left ( \lambda \left ( w-1 \right )t^{\nu } \right ) 
\qquad (\left | w \right |\leq 1, \;\; t>0), 
$$
we have
$$
\mathbb{P}\left \{ H_{j}\leq t \right \}=\frac{\lambda_{j}t^{\nu}}{\Gamma \left ( \nu  \right )}
\int_{0}^{1}\left ( 1-u \right )^{\nu -1}E_{\,\nu ,1}\left ( -\lambda_j \, t^{\nu }u^{\nu }\right )du\\
$$
By making use of formula (2.2.14) of \cite{Mathai2008}, i.e.
\begin{equation*}
\int_{0}^{1}z^{\beta -1}\left ( 1-z \right )^{\sigma -1}E_{\, \alpha ,\beta }\left ( xz^{\alpha }\right )dz=\Gamma \left ( \sigma  \right )E_{\, \alpha ,\sigma +\beta }\left ( x \right ),
\end{equation*}
(where $\alpha>0;\,\beta,\sigma\in\mathbb{C};\,Re(\beta)>0 $ and $ Re(\sigma)>0 $), 
for $ \sigma=\alpha=\nu,\,\beta=1$ and $x=-\lambda_{j}t^{\nu} $, we get
$$
\mathbb{P}\left \{ H_{j}\leq t \right \}
=\lambda_{j}t^{\nu}E_{\,\nu ,\nu +1}\left ( -\lambda _{j}t^{\nu} \right ), \qquad t\geq 0.
$$
Therefore $ H_{j} $ is distributed as the waiting time of the first event of the fractional Poisson process defined in (\ref{PoissonFractional}) (cf.\ (\ref{PDFTempoArrivo})).\qedhere
\end{proof}
\par
The result shown in Theorem \ref{th:WaitingTime} is an immediate extension of the well-known result   
for the Poisson process, i.e.\ for $\nu=1$, by which $H_{j}$ is exponentially distributed with parameter 
$\lambda _{j}$.\\
\par
We will now be concerned with the distribution of the first passage time to a fixed level for the process $ M^{\nu}(t)$, denoted as 
\begin{equation}\label{DefFPT}
 \tau_{n}= \inf \left \{  s>0:M^{\nu }\left ( s \right )= n \right \}, \qquad n \in\mathbb{N}.
\end{equation}
The following result is concerning the case $k=2$, i.e.\ when the process $ M^{\nu}(t) $ performs jumps of sizes 1 and 2.
\begin{theorem}
The cumulative distribution function of the first passage time $ \tau_{k} $ when $k=2$ reads
\begin{align}\label{FPT}
\mathbb{P}\left \{ \tau_{n}\leq t \right \}=&\sum_{h=n}^{+\infty}\sum_{j=\lceil\frac{h}{2}\rceil}^{h}\sum_{i=1}^{j}\binom{i}{n-i}\binom{j-i}{h-n-j+i}\left(\frac{\lambda _{1}}{\lambda _{1}+\lambda _{2}} \right )^{2j-h}\left ( \frac{\lambda _{2}}{\lambda _{1}+\lambda _{2}} \right )^{h-j}\notag\\&\times\int_{0}^{t}\mathbb{P}\left \{ N_{\,\lambda _{1}+\lambda _{2}}^{\,\nu}\left ( t \right )= j,N_{\,\lambda _{1}+\lambda _{2}}^{\,\nu}\left ( s\right )=i \right\}ds, \qquad t> 0.
\end{align}
\end{theorem}
\begin{proof}[Proof]
Since the process $ M^{\nu}(t) $ performs jumps of size 1 and 2, and has non-independent increments, the computation of the cumulative distribution function of the first passage time (\ref{DefFPT}) can be carried out as follows:
\begin{align*}
\mathbb{P}\left \{ \tau_{n}\leq t \right \}=&\sum_{h=n}^{+\infty}\int_{0}^{t}\mathbb{P}\left \{ M^{\nu }\left ( t \right )= h,M^{\nu } \left ( s \right )= n\right \}ds\notag\\
=&\sum_{h=n}^{+\infty}\sum_{j=\lceil\frac{h}{2}\rceil}^{h}\sum_{i=1}^{j}\int_{0}^{t}\mathbb{P}\left \{ M^{\nu }\left ( t \right )= h,M^{\nu } \left ( s \right )= n\mid N_{\,\lambda _{1}+\lambda _{2}}^{\,\nu}\left ( t \right )= j,N_{\,\lambda _{1}+\lambda _{2}}^{\,\nu}\left ( s\right )=i\right \}\notag\\&\times\mathbb{P}\left \{ N_{\,\lambda _{1}+\lambda _{2}}^{\,\nu}\left ( t \right )= j,N_{\,\lambda _{1}+\lambda _{2}}^{\,\nu}\left ( s\right )=i \right \}ds\notag.
\end{align*}
Making use of Proposition \ref{CFPRemark} we have:
\begin{align*}
\mathbb{P}\left \{ \tau_{n}\leq t \right \}=&\sum_{h=n}^{+\infty}\sum_{j=\lceil\frac{h}{2}\rceil}^{h}\sum_{i=1}^{j}\int_{0}^{t}\mathbb{P}\left \{ \sum_{r=1}^{j}X_{r}=h,\sum_{l=1}^{i}X_{l}=n \right \}\notag\\&\times\mathbb{P}\left \{ N_{\,\lambda _{1}+\lambda _{2}}^{\,\nu}\left ( t \right )= j,N_{\,\lambda _{1}+\lambda _{2}}^{\,\nu}\left ( s\right )=i \right \}ds\notag\\
=&\sum_{h=n}^{+\infty}\sum_{j=\lceil\frac{h}{2}\rceil}^{h}\sum_{i=1}^{j}\int_{0}^{t}\mathbb{P}\left \{ \sum_{l=1}^{i}X_{l}=n,\sum_{r=i+1}^{j}X_{r}=h-n \right \}\notag\\&\times\mathbb{P}\left \{ N_{\,\lambda _{1}+\lambda _{2}}^{\,\nu}\left ( t \right )= j,N_{\,\lambda _{1}+\lambda _{2}}^{\,\nu}\left ( s\right )=i \right \}ds\notag\\
=&\sum_{h=n}^{+\infty}\sum_{j=\lceil\frac{h}{2}\rceil}^{h}\sum_{i=1}^{j}\int_{0}^{t}\mathbb{P}\left \{ \sum_{l=1}^{i}X_{l}=n\right \}\mathbb{P}\left \{\sum_{r=i+1}^{j}X_{r}=h-n  \right \}\notag\\&\times\mathbb{P}\left \{ N_{\,\lambda _{1}+\lambda _{2}}^{\,\nu}\left ( t \right )= j,N_{\,\lambda _{1}+\lambda _{2}}^{\,\nu}\left ( s\right )=i \right\}ds\notag\\
=&\sum_{h=n}^{+\infty}\sum_{j=\lceil\frac{h}{2}\rceil}^{h}\sum_{i=1}^{j}\binom{i}{n-i}\left ( \frac{\lambda _{1}}{\lambda _{1}+\lambda _{2}} \right )^{2i-n}\left ( \frac{\lambda _{2}}{\lambda _{1}+\lambda _{2}} \right )^{n-i}\notag\\&\times\binom{j-i}{h-n-j+i}\left(\frac{\lambda _{1}}{\lambda _{1}+\lambda _{2}} \right )^{2j-2i+n-h}\left ( \frac{\lambda _{2}}{\lambda _{1}+\lambda _{2}} \right )^{h-n-j+i}\notag\\&\times\int_{0}^{t}\mathbb{P}\left \{ N_{\,\lambda _{1}+\lambda _{2}}^{\,\nu}\left ( t \right )= j,N_{\,\lambda _{1}+\lambda _{2}}^{\,\nu}\left ( s\right )=i \right\}ds,\notag
\end{align*}
%\begin{align*}
%=&\sum_{h=n}^{+\infty}\sum_{j=\lceil\frac{h}{2}\rceil}^{h}\sum_{i=1}^{j}\binom{i}{n-i}\binom{j-i}{h-n-j+i}\left(\frac{\lambda _{1}}{\lambda _{1}+\lambda _{2}} \right )^{2j-h}\left ( \frac{\lambda _{2}}{\lambda _{1}+\lambda _{2}} \right )^{h-j}\notag\\&\times\int_{0}^{t}\mathbb{P}\left \{ N_{\,\lambda _{1}+\lambda _{2}}^{\,\nu}\left ( t \right )= j,N_{\,\lambda _{1}+\lambda _{2}}^{\,\nu}\left ( s\right )=i \right\}ds,\notag\\
%\end{align*}
%
this giving Eq. (\ref{FPT}).\qedhere
\end{proof}
\-\hspace{8pt}To the best of our knowledge, the bivariate distribution shown in the right-hand-side of (\ref{FPT}), i.e.\ 
$ \mathbb{P}\left \{ N_{\,\lambda _{1}+\lambda _{2}}^{\,\nu}\left ( s\right )=i,N_{\,\lambda _{1}+\lambda _{2}}^{\,\nu}\left ( t \right )= j \right\} $,  cannot be expressed in a closed form. 
\cite{OrsingherPolito2013} derived an expression in terms of Prabhakar integrals, i.e.:
\begin{multline*}
\mathbb{P}\left \{ N_{\,\lambda _{1}+\lambda _{2}}^{\,\nu}\left ( s\right )=i,N_{\,\lambda _{1}+\lambda _{2}}^{\,\nu}\left ( t \right )= j \right\}\\=\left ( \lambda _{1}+\lambda _{2} \right )^{j}\biggl ( \mathbf{E}_{\nu ,\nu i,-(\lambda _{1}+\lambda _{2});\left ( t-s \right )+}^{i} \biggl( \mathbf{E} _{\nu ,\nu(j-i-1)+1,-(\lambda _{1}+\lambda _{2});(z+s-t)+}^{j-i}\\\times y^{\nu -1}E_{\nu ,\nu }(-\left ( \lambda _{1} +\lambda _{2}\right )y^{\nu })\biggr )(z) \biggr)(t),
\end{multline*}
where
\begin{equation*}
\biggl ( \mathbf{E}_{\rho ,\mu ,\omega ;a +}^{\gamma   }\phi \biggr )\left ( x \right )= \int_{a}^{x}\left ( x-t \right )^{\mu -1}E_{\rho ,\mu }^{\gamma }\left ( \omega \left ( x-t \right )^{\rho } \right )\phi \left ( t \right )dt
\end{equation*}
is the Prabhakar integral (see \cite{Prabhakar1971} for details).
Politi  {\em et al.}\ \cite{Politi2011}, instead, evaluate the joint probability given in (\ref{FPT}) by introducing the random variable $ Y_{i} $ which denotes the residual lifetime at $ s $ (that is the time to the next epoch) conditional on $ N_{\,\lambda _{1}+\lambda _{2}}^{\,\nu}\left ( s \right )= i $, i.e. $ Y_{i}\overset{def}{=}\left [ \tau_{i}-s\mid N_{\,\lambda _{1}+\lambda _{2}}^{\,\nu}\left ( s \right )= i \right ] $
whose cumulative distribution function is denoted by $ F_{Y_{i}}(y) $.  Therefore,
\begin{align*}
\mathbb{P}\left \{ N_{\,\lambda _{1}+\lambda _{2}}^{\,\nu}\left ( s\right )=i,N_{\,\lambda _{1}+\lambda _{2}}^{\,\nu}\left ( t \right )= j \right\}=&\mathbb{P}\left \{ N_{\,\lambda _{1}+\lambda _{2}}^{\,\nu}\left ( t\right )-N_{\,\lambda _{1}+\lambda _{2}}^{\,\nu}\left ( s \right )= j-i\mid N_{\,\lambda _{1}+\lambda _{2}}^{\,\nu}\left ( s\right )=i\right\}\notag\\&\times\mathbb{P}\left \{N_{\,\lambda _{1}+\lambda _{2}}^{\,\nu}\left ( s\right )=i\right\},\notag
\end{align*}
where
\begin{multline*}\mathbb{P}\left \{ N_{\,\lambda _{1}+\lambda _{2}}^{\,\nu}\left ( t\right )-N_{\,\lambda _{1}+\lambda _{2}}^{\,\nu}\left ( s \right )= j-i\mid N_{\,\lambda _{1}+\lambda _{2}}^{\,\nu}\left ( s\right )=i\right\}\\=
\begin{cases}
\int_{0}^{t-s}\mathbb{P}\left \{ N_{\,\lambda _{1}+\lambda _{2}}^{\nu } \left ( t-s-y \right )= j-i-1\right \}dF_{Y_{i}}(y),&\text{if $ j-i\geq 1 $},\\1-F_{Y_{i}}(t-s),&\text{if $ j-i=0 $}.
\end{cases}
\end{multline*}
It is meaningful to stress that when $k=2$ the passage of $ M^{\nu}(t) $ to a level $ n $ is not sure. 
In fact, the process can cross state $ n $ without visiting it due to the effect of a jump having size 2.
%%%%%%%%%%%%%%%%%%%%%%
\section{Convergence results}\label{Section5}
%%%%%%%%%%%%%%%%%%%%%%
For the processes $ N_{\,\lambda}^{\,\nu}(t) $ and $ M^{\nu}(t) $, defined respectively in (\ref{PoissonFractional}) and in (\ref{FracPoisson2}), we now focus on a property related to their asymptotic behavior as the relevant parameters grow larger. 
\newtheorem{Cutoff}{Proposition}[section]
\begin{Cutoff}\label{Cutoff}
Let $ \nu \in (0,1]$. Then for a fixed $t>0$ we have 
\begin{equation*}
\frac{N_{\, \lambda}^{\,\nu}(t)}{\mathbb{E}\left [ N_{\, \lambda}^{\,\nu}(t) \right ]}\xrightarrow[\lambda \rightarrow +\infty]{\mathrm{Prob}}1.
\end{equation*}
\end{Cutoff}
\begin{proof}[$ \mathbf{Proof} $]
We study the convergence in mean of the random variable $ \frac{N_{\, \lambda}^{\,\nu}(t)}{\mathbb{E}\left [ N_{\, \lambda}^{\,\nu}(t) \right ]} $ to 1. Due to the triangular inequality we have 
\begin{align*}
\mathbb{E}\left [ \left |\frac{N_{\, \lambda}^{\,\nu}(t)}{\mathbb{E}\left [ N_{\, \lambda}^{\,\nu}(t) \right ]} -1 \right | \right ]\leq 2.
\end{align*}
Therefore, we can apply the dominated convergence theorem and calculate the following limit:
\begin{align}\label{limit}
\lim_{\lambda \rightarrow +\infty}\mathbb{E}\left [ \left |\frac{N_{\, \lambda}^{\,\nu}(t)}{\mathbb{E}\left [ N_{\, \lambda}^{\,\nu}(t) \right ]} -1 \right | \right ]=&\lim_{\lambda \rightarrow +\infty}\sum_{j=0}^{+\infty}\left | \frac{j}{\frac{\lambda t^{\nu}}{\Gamma \left ( \nu+1 \right )}}-1 \right | % \notag\\&\times
\left (\lambda t^{\nu }  \right )^{j}E_{\nu,j\nu +1 }^{j+1 }(-\lambda t^{\nu }).
\end{align} 
Taking account of the behavior of the generalized Mittag-Leffler function for large $ z $ (see \cite{Saxena2004} for details), i.e.:
\begin{equation*}
E_{\alpha,\beta }^{\delta }(z)\sim {\mathcal O} \left ( \left | z \right |^{-\delta } \right ),\qquad \left | z \right |>1,
\end{equation*}
we can conclude that limit (\ref{limit}) equals $ 0 $. This fact proves the result since convergence in mean implies convergence in probability.\qedhere
\end{proof}
The previous result can be extended to a more general setting.  Recalling 
the expression (\ref{moments}) for the moments of $N_{\, \lambda }^{\, \nu }\left ( t \right )$,
the proof of the next proposition is similar to that of Proposition \ref{Cutoff} and thus is omitted. 
\newtheorem{CutoffGeneral}[Cutoff]{Proposition}
\begin{CutoffGeneral}\label{CutoffGeneral}
Let $ \nu \in (0,1]$ and $ r\in\mathbb{N} $. Then, for a fixed $t>0$, 
\begin{equation*}
\frac{\left [ N_{\, \lambda}^{\,\nu}(t) \right ]^{r}}{\mathbb{E}\left \{ \left [ N_{\, \lambda}^{\,\nu}(t) \right ]^{r} \right \}}\xrightarrow[\lambda \rightarrow +\infty]{\mathrm{Prob}}1.
\end{equation*}
\end{CutoffGeneral}
In order to prove an analogous result for $ M^{\nu}(t) $, in the following lemma we give a formal expression for the moments of such a process.
\newtheorem{Lem}{Lemma}[section]
\begin{Lem}\label{lemma}
The $ m^{th} $ order moment of the process $ M^{\nu}(t)$, $t\geq 0$, reads
\begin{align}\label{Moments}
\mathbb{E}\left \{ \left [ M^{\,\nu}(t) \right ]^{m} \right \}
&=\sum_{r=0}^{m}\frac{t^{r\nu}}{\Gamma \left ( r\nu+1 \right )}
\sum_{i_{1}+\ldots +i_{k}= r}\binom{r}{i_{1},\ldots, i_{k}}\lambda _{1}^{i_{1}}\ldots \lambda _{k}^{i_{k}}
\nonumber \\
 &\;\;\times\sum_{n_{1}+\ldots +n_{k}= m}\binom{m}{n_{1},\ldots, n_{k}}
 \left [ \frac{d^{\, n_{1}} }{ds^{\, n_{1}}}\big ( e^{s}-1 \big )^{i_{1}}\ldots  \frac{d^{\, n_{k}} }{ds^{\, n_{k}}}
 \big ( e^{ks}-1 \big )^{i_{k}}\right ]\bigg |_{s=0}.
\end{align}
\end{Lem}
\begin{proof}[$ \mathbf{Proof} $]
By applying Hoppe's formula in order to evaluate the derivatives of the moment generating function of the process $ M^{\,\nu}(t) $, cf.\ (\ref{fgmomM}), we have
\begin{equation*}
\mathbb{E}\left \{ \left [ M^{\,\nu}(t) \right ]^{m} \right \}=\sum_{r=0}^{m}\frac{\left ( E_{\nu,1} (z)\right )^{\left ( r \right )}\big |_{z=\sum_{j=1}^{k}\lambda_{j}\left ( e^{js}-1 \right )t^{\nu }}}{r!}
\,A_{m,r}\Bigg ( \sum_{j=1}^{k}\lambda_{j}\left ( e^{js}-1 \right )t^{\nu } \Bigg )\Bigg |_{s=0},
\end{equation*}
where
\begin{align*}
A_{m,r}\Bigg ( \sum_{j=1}^{k}\lambda_{j}\left ( e^{js}-1 \right )t^{\nu } \Bigg )
=&\sum_{h=0}^{r}\binom{r}{h}\Bigg ( -\sum_{j=1}^{k}\lambda_{j}\left ( e^{js}-1 \right )t^{\nu } \Bigg )^{\!r-h}
\\
&\times \frac{d^{m}}{ds^{m}}\Bigg ( \sum_{j=1}^{k}\lambda_{j}\left ( e^{js}-1 \right )t^{\nu } \Bigg )^{\!h}.
\end{align*}
Finally, after using rather cumbersome algebra, we obtain (\ref{Moments}).
\end{proof}
It is now immediate to verify the following result for  $ M^{\nu}(t) $.
\newtheorem{CutoffPoisson2}[Cutoff]{Proposition}
\begin{CutoffPoisson2}\label{CutoffPoisson2}
Let $ \nu \in (0,1]$ and $ m\in\mathbb{N} $. Then, for $ i\in\left \{ 1,2,\ldots, k \right \} $ and 
for a fixed $t>0$, we have
\begin{equation*}
\frac{\left [ M^{\,\nu}(t) \right ]^{m}}{\mathbb{E}\left \{ \left [ M^{\,\nu}(t) \right ]^{m} \right \}}\xrightarrow[\lambda_{i} \rightarrow +\infty]{\mathrm{Prob}}1.
\end{equation*}
\end{CutoffPoisson2}
\begin{proof}[$ \mathbf{Proof} $]
By virtue of (\ref{Moments}), convergence in probability can be obtained by proving convergence in mean, as in Proposition \ref{Cutoff}. \qedhere
\end{proof}
The results presented in this section deserve interest in some physical contexts. We recall that a family of random variables $ U^{(\lambda)} $ exhibits \textit{cut-off behavior} at mean times if (see, for instance, Definition 1 of \cite{Barrera2009})
\begin{equation*}
\frac{U^{\left ( \lambda  \right )}}{\mathbb{E}\left [ U^{\left ( \lambda  \right )} \right ]}\xrightarrow[\lambda \rightarrow +\infty]{\mathrm{Prob}}1.
\end{equation*}
Hence, Propositions \ref{Cutoff}, \ref{CutoffGeneral} and \ref{CutoffPoisson2} show that the processes $ \left [N_{\lambda }^{\nu }\left ( t \right )  \right ]^{m} $ and $ \left [M^{\nu}\left ( t \right )  \right ]^{m} $, $ m\in\mathbb{N} $, exhibit cut-off behavior at mean times with respect to the relevant parameters or, roughly speaking, that they somehow converge very abruptly to equilibrium.
\\\\
We finally remark that in this context the sufficient condition given in Proposition 1 of \cite{Barrera2009} is not useful to prove Proposition \ref{Cutoff}, since such condition holds only when $ \nu=1 $.

\section*{Acknowledgements}
The authors would like to thank an anonymous referee for some useful comments.

\bibliographystyle{alea3}
%\begin{thebibliography}{99}

\bibliography{ }

\end{document}